\documentclass[a4paper, 10pt]{article}
\usepackage[latin1]{inputenc}
\usepackage{amsmath}
\usepackage{amssymb}
\usepackage{amsthm}
\usepackage[all]{xy}
\usepackage[english]{babel}

\title{Solution to the Clebsch-Gordan problem for string algebras}
\author{Martin Herschend}

\newcommand{\edge}[1]{\ar@{-}[#1]}
\newcommand{\iso}{\;\tilde\rightarrow\;}
\newcommand{\arr}[3]{#1\overset{#2}{\rightarrow} #3}

\newcommand{\cat}[1]{\mathop{\rm #1}\nolimits}

\newcommand{\eg}{e.g. }
\newcommand{\ie}{i.e. }
\def\b{\mbox{\tiny$\bullet$}}

\newtheorem{thm}{Theorem}
\newtheorem{cor}{Corollary}
\newtheorem{pro}{Proposition}
\newtheorem{lem}{Lemma}

\theoremstyle{definition}

\newtheorem{exmp}{Example}
\begin{document}
\date{}
\maketitle

\begin{abstract}
\noindent 
The category of modules over a string algebra is equipped with a tensor product defined point-wise and arrow-wise in terms of the underlying quiver. In the present article we investigate how this tensor product interacts with the classification of indecomposables. We apply the results obtained to solve the Clebsch-Gordan problem for string algebras. Moreover, we describe the corresponding representation ring and tensor ideals in the module category.
\end{abstract}


\section{Introduction}
In their classification of Harish-Chandra modules over the Lorentz group \cite{gelfand68}, Gelfand and Ponomarev encountered and solved the following problem in linear algebra: Find a canonical form for nilpotent linear operators $a$ and $b$ satisfying $ab = ba = 0$. This problem is equivalent to classifying all indecomposable finite dimensional modules over the algebra $\Lambda$, generated by two elements $\alpha$ and $\beta$ subject to the relations $\alpha^n = \beta^n = \alpha\beta = \beta\alpha = 0$. As they showed the indecomposable $\Lambda$-modules are of two kinds, later called strings and bands.

Their techniques, including the classification in terms of strings and bands, was later shown to apply in wider generality by Ringel \cite{ringel75} and Donovan-Freislich \cite{donovan78}. The related Auslander-Reiten theory has also been investigated by Butler-Shahzamanian in \cite{butler80} and by Butler-Ringel in \cite{butler87}.

A quite general situation in which these techniques apply, is that of special biserial algebras. The finite type special biserial algebras were investigated by Skowronski-Wasch\"usch in \cite{skowronski83}, and the tame by Wald-Waschb\"usch in \cite{wald83}. An important feature of special biserial algebras is that they commonly appear as tame or finite type blocks of modular group algebras \cite{erdmann90}. 

String algebras are a slight specialisation of special biserial algebras. Although many of the results in the present article apply to special biserial algebras, we restrict our investigation to string algebras for sake of clarity.

The Clebsch-Gordan problem originates from the investigation of binary algebraic forms by Clebsch and Gordan \cite{clebsch66}. Today the most common incarnation of this problem is the following: Given two indecomposable representations of a group $G$ decompose their tensor product into a direct sum of indecomposables. However, the Clebsch-Gordan problem can be posed for any Krull-Schmidt category equipped with a tensor product. One such category is that of representations of a quiver with semi-monomial relations. It has a tensor product defined point-wise and arrow-wise \cite{herschend07},\cite{herschend07a}.

This version of the Clebsch-Gordan problem was first considered for the loop quiver. If the ground field is algebraically closed, the Clebsch-Gordan problem for the loop is equivalent to finding the Jordan normal form of the Kronecker product of any two Jordan blocks. It was solved in characteristic zero by Aitken \cite{aitken35}. Other independent solutions are due to Huppert \cite{huppert90} and Martsinkovsky-Vlassov \cite{MaVl}. In positive characteristic there is no known formula for the decomposition. However, algorithms for finding the decomposition exist \cite{iwamatsu07a}. Parts of the loop case turns out to be important for our investigation and we treat it briefly in the next section.

In finite type the Clebsch-Gordan problem has been solved for quivers of Dynkin type $\mathbb{A}$, $\mathbb{D}$ and $\mathbb{E}_6$ in \cite{herschend07b} and \cite{herschend08}. Known solutions in tame type exist for extended Dynkin quivers of type $\tilde{\mathbb{A}}$ \cite{herschend05} and for $\Lambda$ defined as above \cite{herschend07}. 

In this paper we solve the Clebsch-Gordan problem for the quivers with relations that correspond to string algebras. It should be noted that this class includes the tame cases mentioned above. Thus our results form a quite drastic generalisation of previously known solutions for tame type.

Theorem \ref{CG} presents the solution to the Clebsch-Gordan problem as a list of formulae describing the decomposition of the tensor product of any pair of indecomposables. Qualitatively, this result can be described as follows:
\begin{itemize}
\item A string tensor a string or band is isomorphic to a direct sum of strings.
\item The tensor product of two bands with non-isomorphic shapes is also isomorphic to a direct sum of strings.
\item The tensor product of two bands with the same shape is isomorphic to a direct sum of bands with the same shape and certain strings.
\end{itemize} 
To determine exactly which strings appear in the right hand side of these formulae is an easy combinatorial task in each specific case. However, some aspects of the general picture may be elusive. Therefore, we adopt certain alternative viewpoints on the problem to obtain more qualitative information. To be precise, Theorem \ref{ring} describes the representation ring, which has the Clebsch-Gordan formulae encoded in its multiplicative structure, and Theorem \ref{tensideal} describes the tensor ideals in the representation category. 

\subsection*{Acknowledgements}
This paper was written during my stay at Nagoya University, which was financed by the Japan Society for the Promotion of Science. I am grateful for their support. I would also like to thank my host Osamu Iyama for his hospitality and our many interesting discussions. I would also like to thank Ryan Kinser for pointing out a typo in an earlier version of the paper.

\section{Preliminaries}\label{prelim}
Throughout, fix an algebraically closed field $\Bbbk$. In this section we present some notation and terminology, much of which can be found in \cite{gabriel92}. A category $\mathcal{A}$ is called linear if the space of morphisms from $x$ to $y$, denoted $\mathcal{A}(x,y)$, is a vector space over $\Bbbk$ for all $x,y \in \mathcal{A}$ and all composition maps are bilinear. Let $\mathcal{A}$ and $\mathcal{B}$ be linear categories. A functor $F: \mathcal{A} \rightarrow \mathcal{B}$ is called linear if the induced maps between morphism spaces are linear.

A module over a small linear category $\mathcal{A}$ is a linear functor $m : \mathcal{A} \rightarrow \cat{Vec} \Bbbk$, where $\cat{Vec} \Bbbk$ denotes the category of vector spaces over $\Bbbk$. It is called (point-wise) finite dimensional if $\dim m (x) < \infty$ for all $x \in \mathcal{A}$. Morphisms of $\mathcal{A}$-modules are natural transformations. We denote the category of $\mathcal{A}$-modules by $\cat{Mod} \mathcal{A}$ and the full subcategory of all finite dimensional modules by $\cat{mod} \mathcal{A}$. Both of these categories are equipped with direct sums, taken object-wise and morphism-wise. In the sequel we assume that all modules are finite dimensional.

Let $\mathcal{A}$ be a small linear category. An ideal in $\mathcal{A}$ is a family of subspaces $\mathcal{I}(x,y) \subset \mathcal{A}(x,y)$, where $x,y \in \mathcal{A}$ such that $\mathcal{A}(y,z)\mathcal{I}(x,y)\mathcal{A}(w,x) \subset \mathcal{I}(w,z)$. The factor category $\mathcal{A}/\mathcal{I}$ has the same objects as $\mathcal{A}$ and morphism spaces defined by $(\mathcal{A}/\mathcal{I})(x,y) = \mathcal{A}(x,y)/\mathcal{I}(x,y)$. As is usual we identify $\cat{mod}\mathcal{A}/\mathcal{I}$ with the full subcategory of $\cat{mod}\mathcal{A}$ consisting of modules $m$ satisfying $m(\gamma) = 0$ for all $\gamma \in \mathcal{I}$.

To any linear functor $F: \mathcal{A} \rightarrow \mathcal{B}$ we associate its pull-up functor
\[
F^* : \cat{mod}\mathcal{B} \rightarrow \cat{mod}\mathcal{A}
\]
defined by $F^* m = m \circ F$ for all $\mathcal{B}$-modules $m$ and $(F\phi)_a = \phi_{Fa}$ for all $\mathcal{B}$-module morphisms $\phi$ and $a \in \mathcal{A}$. 

A quiver $Q$ consists of a set of vertices $Q_0$ and a set of arrows $Q_1$ together with two maps $t,h: Q_1 \rightarrow Q_0$ mapping each arrow $\alpha$ to its tail $t \alpha $ and head $h \alpha$ respectively. We denote the set of arrows from $x$ to $y$ by $Q_1(x,y)$ and write $\arr x \alpha y$ to indicate that $\alpha \in Q_1(x,y)$. We assume that all quivers are finite, \ie that $Q_0$ and $Q_1$ are finite sets.

A morphism from a quiver $P$ to a quiver $Q$ consists of two maps $f_0:P_0 \rightarrow Q_0$ and $f_1:P_1 \rightarrow Q_1$ such that $f_0(g \alpha) = g f_1(\alpha)$ for $g \in \{t,h\}$.

A path of length $d > 0$ in $Q$ from $x$ to $y$ is a sequence of arrows $\alpha_d\cdots \alpha_1$ such that $t \alpha_ 1 = x$, $h \alpha_i = t \alpha_{i+1}$ and $h \alpha_d = y$ for all $1 \leq i < d$. For each vertex $x$ there is also a path $e_x$ of length $0$ starting and ending at $x$. We denote the set of paths from $x$ to $y$ by $Q(x,y)$. The linear path category $\Bbbk Q$ is defined by the following properties. The objects in $\Bbbk Q$ are the vertices of $Q$. For all $x,y \in Q_0$ the space of morphisms from $x$ to $y$ has $Q(x,y)$ as basis. Composition of paths is given by concatenation.

Any $\Bbbk Q$-module $m$ is determined by the collection of vector spaces $(m(x))_{x \in Q_0}$ and linear maps $(m(\alpha))_{\alpha \in Q_1}$. Such a collection of vector spaces and linear maps is called a representation of $Q$. In fact this correspondence determines an equivalence between $\cat{mod} \Bbbk Q$ and the category of representations of $Q$.

Let $m$ and $n$ be $\Bbbk Q$-modules. Their tensor product is defined by $(m \otimes n)(x) = m(x) \otimes n(x)$ for all $x \in Q_0$ and $(m \otimes n)(\alpha) = m(\alpha) \otimes n(\alpha)$ for all $\alpha \in Q_1$. Since the tensor product is defined point-wise and arrow-wise it is exact and commutes with direct sums. Moreover, it is commutative: $m \otimes n \iso n \otimes m$. 

Let $\mathcal{I}$ be an ideal in $\Bbbk Q$. The tensor product $m \otimes n$ of two $\Bbbk Q/ \mathcal{I}$-modules $m$ and $n$ is a $\Bbbk Q$-module but not necessarily a $\Bbbk Q/ \mathcal{I}$-module since it may happen that $(m \otimes n)(\gamma) \not = 0$ for some $\gamma \in \mathcal{I}$. We say that  $\mathcal{I}$ is semi-monomial if it is generated by zero relations and commutativity relations. In this case $m \otimes n$ is a $\Bbbk Q/ \mathcal{I}$-module for all $\Bbbk Q/ \mathcal{I}$-modules $m$ and $n$, as is shown in \cite{herschend07}. The Clebsch-Gordan problem for $\Bbbk Q/ \mathcal{I}$ is the problem of finding the decomposition of $u \otimes v$ for all indecomposable $\Bbbk Q/ \mathcal{I}$-modules $u$ and $v$. We quote the following proposition from \cite{herschend08}.

\begin{pro}\label{pullup}
Let $P$, $Q$ be quivers and $F : \Bbbk P \rightarrow \Bbbk Q$ be a linear functor that sends paths to paths. Then the induced pull-up functor
\[
F^* : \cat{mod} \Bbbk Q \rightarrow \cat{mod} \Bbbk P
\]
respects direct sums and tensor products in the sense that
\[
F^*(m \oplus n) = (F^*m) \oplus (F^*n)  
\]
and
\[
F^*(m \otimes n) = (F^*m) \otimes (F^*n).
\]
\end{pro}

Let $Q$ be a quiver. To each subquiver $Q'$ of $Q$ we associate the characteristic representation $\chi_{Q'}^{Q} \in \cat{mod} \Bbbk Q$ defined by 
\[
\chi_{Q'}^{Q}(x) =  \begin{cases}
\Bbbk & \mbox{if } x \in Q'_0 \\ 
0 & \mbox{if } x \not\in Q'_0 \end{cases}
\]
and
\[
\chi_{Q'}^{Q}(\alpha) =  \begin{cases}
1_\Bbbk & \mbox{if } \alpha \in Q'_1 \\ 
0 & \mbox{if } \alpha \not\in Q'_1 \end{cases}
\]
The module $\chi_Q = \chi_Q^Q$ is a tensor identity in the sense that $\chi_Q \otimes m  \iso m$ for all $m \in \cat{mod}\Bbbk Q$. Moreover, any linear functor $F : \Bbbk P \rightarrow \Bbbk Q$ sending paths to paths satisfies $F^*\chi_Q \iso \chi_P$.

Let $Z$ be the double loop quiver, \ie $Z$ has one vertex $x$ and two arrows $\alpha, \beta$ (as a consequence $t \alpha = h \alpha =t \beta = h \beta = x$). Further, let $\mathcal{I}$ be the ideal defined by the relations $\alpha\beta = \beta\alpha = e_x$ and set $\Gamma = \Bbbk Z / \mathcal{I}$. We call $\Gamma$ the Laurent algebra. Each $\Gamma$-module $m$ is determined by the linear automorphism $m(\alpha)$, since $m(\beta) = m(\alpha)^{-1}$. As $\Bbbk$ is algebraically closed the indecomposable $\Gamma$-module are classified by the Jordan normal form. More precisely, for each $\lambda \in \Bbbk^{\iota} = \Bbbk \setminus \{0\}$ and positive integer $s$, let $b_{\lambda,s}$ be the $\Gamma$-module satisfying $b_{\lambda,s}(x) = \Bbbk^s$ and $b_{\lambda,s}(\alpha) = J_{\lambda}(s)$, where $J_{\lambda}(s)$ has the matrix
\[
\left[ \begin{array}{cccc} \lambda & 1 & & \\ & \ddots &\ddots &\\ & & \lambda & 1 \\ & & & \lambda \end{array} \right]
\]
in the standard basis of $\Bbbk^s$. Then the modules $b_{\lambda,s}$ classify all indecomposable $\Gamma$-modules.

The following theorem due to Aitken \cite{aitken35} solves the Clebsch-Gordan problem for $\Gamma$ in characteristic zero. In prime characteristic the solution is harder to describe (see \cite{iwamatsu07a} and \cite{herschend08a}).

\begin{thm}\label{jordan}
Let $\Bbbk$ be a field of characteristic zero. For all $\lambda,\mu \in \Bbbk^{\iota}$ and positive integers $s,t$, the following formula holds:
\[
J_{\lambda}(s) \otimes J_{\mu}(t)\iso \bigoplus_{k=0}^{min(s,t)-1}J_{\lambda\mu}(s+t-2k-1)
\]
\end{thm}

The solution to the Clebsch-Gordan problem for string algebras turns out to depend on the solution for $\Gamma$. To simplify notation we define the non-increasing sequence of natural numbers $(l_k)_{k \in I_{st}}$ by
\[
J_{\lambda}(s) \otimes J_{\mu}(t)\iso \bigoplus_{k \in I_{st}} J_{\lambda\mu}(l_k).
\]
For dimension reasons we have $\sum_{k \in I_{st}} l_k = st$. If $\Bbbk$ has characteristic zero, then $(l_k)_{k \in I_{st}} = (s+t-1, s+t-3, \ldots, |s-t|+1)$ by Theorem \ref{jordan}.

For non-algebraically closed fields $\Gamma$-modules can be classified in terms of irreducible polynomials. In fact, replacing Jordan blocks with indecomposable automorphisms in the definition of bands gives a classification of modules over string algebras even in the non-algebraically closed case \cite{butler87}. The solution to the Clebsch-Gordan problem contained in the present article can also be adapted to the general setting and thus is reduced to the Clebsch-Gordan problem for $\Gamma$-modules. This problem has a general description for prefect fields that can be made explicit for real closed fields \cite{herschend08a}.

We will sometimes use the following matrix notation to define linear maps. Let $\{V_i\}_{i \in I}$ and $\{W_j\}_{j \in J}$ be two families of vector spaces. Further let $\iota_i : V_i \rightarrow\bigoplus_{i \in I} V_i$ be the canonical inclusion and $\pi_j : \bigoplus_{j \in J} W_j \rightarrow W_j$ the canonical projection. For any linear map $A : \bigoplus_{i \in I} V_i \rightarrow \bigoplus_{j \in J} W_j$ and $(i,j) \in I \times J$ set $A_{ji} = \pi_jA\iota_i : V_i \rightarrow W_j$. 

To reduce the amount of parentheses in our notation we declare that functors have higher precedence than the tensor product, which in turn has higher precedence than the direct sum. Thus $Fm \otimes n = (Fm) \otimes n$ and $m \otimes n \oplus u = (m \otimes n) \oplus u$.

We denote the disjoint union of two sets $X$ and $Y$ by $X \amalg Y$ and the cardinality of $X$ by $|X|$.

\section{Methods}
In this section we develop methods to define strings and bands in a uniform way that is compatible with the tensor product. Let $P$ and $Q$ be quivers. A functor $F: \Bbbk P \rightarrow \Bbbk Q$ is called a wrapping if it induces an injective map $P_1(x,y) \rightarrow Q_1(Fx, Fy)$, $\alpha \mapsto F\alpha$ for all $x,y \in P_0$. In particular $F$ is induced by a quiver morphism $P \rightarrow Q$.

For a wrapping $F: \Bbbk P \rightarrow \Bbbk Q$, define the push-down functor
\[
F_{*}: \cat{Mod} \Bbbk P \rightarrow \cat{Mod} \Bbbk Q 
\]
as follows. For each vertex $x' \in Q_0$ set
\[
(F_{*}m)(x') = \bigoplus_{Fx=x'}m(x).
\]
For each arrow $\arr {x'}{\alpha'}{y'}$ in $Q$ let
\[
(F_{*}m)(\alpha') : \bigoplus_{Fx=x'}m(x) \rightarrow \bigoplus_{Fy=y'}m(y)
\]
be the linear map $A$ satisfying  
\[
A_{yx} = \begin{cases}
m(\alpha) & \mbox{if }  \alpha \in P_1(x,y) \mbox{ and } F(\alpha) = \alpha'\\ 
0 & \mbox{otherwise.}\end{cases}
\]
It is well-defined since $F$ is a wrapping. If $\mathcal{I}$ is an ideal in $\Bbbk Q$, then we say that $F$ is compatible with $\mathcal{I}$ if all elements in $\mathcal{I}$ can be written as linear combinations of paths not of the form $F\gamma$, where $\gamma$ runs through all paths in $P$. In this case the push-down $F_{*}$ is also well-defined as a functor from $\cat{Mod} \Bbbk P$ to $\cat{Mod}\Bbbk Q / \mathcal{I}$.

Let $P,Q,R$ be quivers and $F : \Bbbk P \rightarrow \Bbbk Q$, $G : \Bbbk Q \rightarrow \Bbbk R$ wrappings. Then $GF$ is also a wrapping. Moreover, $(GF)^* = F^*G^*$ and $(GF)_* =G_*F_*$. If $P$ is a subquiver of $Q$ and $F$ is given by inclusion, then $GF$ is called the restriction of $G$ to $P$ and denoted by $G |_P$. The push-down functor respects direct sums. With respect to the tensor product we can say the following.

\begin{pro}\label{tens}
Let $m \in \cat{mod} \Bbbk P$, $n \in \cat{mod} \Bbbk Q$ and $F: \Bbbk P \rightarrow \Bbbk Q$ be a wrapping. Then
\[
F_*m \otimes n = F_*(m \otimes F^* n).
\]
\end{pro}
\begin{proof}
For each $x' \in Q_0$ we have
\[\begin{split}
F_*(m \otimes F^* n)(x') &= \bigoplus_{Fx=x'}m(x) \otimes (F^*n)(x) = \bigoplus_{Fx=x'}m(x) \otimes n(Fx)\\
=  \bigoplus_{Fx=x'}m(x) \otimes n(x')& = \left(\bigoplus_{Fx=x'}m(x)\right) \otimes n(x') = (F_*m)(x') \otimes n(x').
\end{split}\]
Now let $\arr {x'}{\alpha'}{y'}$ be an arrow in $Q$,
\[
A = F_*(m \otimes F^* n)(\alpha') : \bigoplus_{Fx=x'}m(x) \otimes n(x') \rightarrow \bigoplus_{Fy=y'}m(y) \otimes n(y')
\]
and 
\[
B = (F_*m)(\alpha') : \bigoplus_{Fx=x'}m(x) \rightarrow \bigoplus_{Fy=y'}m(y).
\]
If there is $\alpha \in P_1(x,y)$ such that $F(\alpha) = \alpha'$, then
\[
A_{yx} = (m \otimes F^* n)(\alpha) = m(\alpha) \otimes n(\alpha') = B_{yx} \otimes n(\alpha').
\]
Otherwise $A_{yx} = 0 = B_{yx} \otimes n(\alpha')$. Hence $F_*(m \otimes F^* n)(\alpha') = (F_*m \otimes n)(\alpha')$.
\end{proof}

Let $P^1,P^2$ and $Q$ be quivers. Furthermore, let $F_1: \Bbbk P^1 \rightarrow \Bbbk Q$ and $F_2: \Bbbk P^2 \rightarrow \Bbbk Q$ be wrappings. We define the fibre product $P^1 \times P^2$ to be the quiver $P$, where $P_0 = \{ (x',x) \in P^1_0 \times P^2_0 \;|\; F_1x' = F_2x \}$, $P_1 = \{ (\alpha',\alpha) \in P^1_1 \times P^2_1 \;|\; F_1\alpha' = F_2\alpha \}$ and $g(\alpha',\alpha) = (g\alpha',g\alpha)$ for $g \in \{t,h\}$. It is endowed with two wrappings $\Pi_1 : \Bbbk P \rightarrow \Bbbk P^1$ and $\Pi_2 :\Bbbk P \rightarrow \Bbbk P^2$, defined by taking the first, respectively second component of the pairs of vertices and arrows. Thus we obtain a commutative diagram
\[
\xymatrix{ & \Bbbk P \ar[dl]_{\Pi_1}\ar[dr]^{\Pi_2} & \\
\Bbbk P^1 \ar[dr]_{F_1} & & \Bbbk P^2\ar[dl]^{F_2}\\ 
& \Bbbk Q &
}
\]
The following proposition relates the fibre product to pull-ups and push-downs.

\begin{pro}\label{com} Let $F_1: \Bbbk P^1 \rightarrow \Bbbk Q$ and $F_2: \Bbbk P^2 \rightarrow \Bbbk Q$ be wrappings. Then $F^*_1F_{2*} = \Pi_{1*}\Pi_2^*$.
\end{pro}
\begin{proof}
Set $P = P^1\times P^2$. Let $m \in \cat{mod} \Bbbk P^2$ and $x' \in P^1_0$. Then
\[
(F^*_1F_{2*} m)(x') = (F_{2*} m )(F_1 x') = \bigoplus_{F_2x = F_1x'} m(x)
\]
and
\[
(\Pi_{*1}\Pi^*_2 m)(x') =  \bigoplus_{\Pi_1z = x'} m(\Pi_2 z).
\]
However $\{ z \in P_0 \;|\; \Pi_1z = x' \} = \{ (x',x) \;|\; x \in P^2_0, \; F_1x' = F_2x \}$ and thus $(F^*_1F_{2*} m)(x') = (\Pi_{*1}\Pi^*_2 m)(x')$.

Let $\arr {x'}{\alpha'}{y'}$ be an arrow in $P^1$,
\[
A = (F^*_1F_{2*} m)(\alpha') : \bigoplus_{F_2x = F_1x'} m(x) \rightarrow \bigoplus_{F_2y = F_1y'} m(y)
\]
and
\[
B = (\Pi_{*1}\Pi^*_2 m)(\alpha') : \bigoplus_{F_2x = F_1x'} m(x) \rightarrow \bigoplus_{F_2y = F_1y'} m(y).
\]
Let $x,y \in P^2_0$. If there is $\alpha \in P^2_1(x,y)$ such that $F_2(\alpha) = F_1(\alpha')$, then $A_{yx} = m(\alpha)$. Moreover, the arrow $(\alpha',\alpha) \in P_1((x',x),(y',y))$ is mapped to $\alpha'$ by $\Pi_1$. Hence
\[
B_{yx} = (\Pi^*_2 m)(\alpha',\alpha) = m(\alpha) = A_{yx}.
\]
Otherwise there is no $\alpha \in P^2_1$ such  $(\alpha,\alpha') \in P_1((x,x'),(y,y'))$ and thus $A_{yx} = 0 = B_{yx}$. Hence $(F^*_1F_{2*} m)(\alpha') = (\Pi_{*1}\Pi^*_2 m)(\alpha')$.
\end{proof}
By symmetry we obtain $F^*_2F_{1*} = \Pi_{2*}\Pi_1^*$. It is worth noting that $\Pi_2^*F_2^* = (F_2\Pi_2)^* = (F_1\Pi_1)^* = \Pi_1^*F_1^*$ and similarly $F_{2*}\Pi_{2*}  = (F_2\Pi_2)_* = (F_1\Pi_1)_* = F_{1*}\Pi_{1*}$. Now we combine this result with the tensor product.

\begin{cor}\label{lift}
Let $m \in \cat{mod} \Bbbk P^1$ and $n \in \cat{mod} \Bbbk P^2$. Then
\[
F_{1*} m \otimes F_{2*} n = F_{1*}\Pi_{1*} (\Pi_1^* m \otimes \Pi_2^* n).
\]
\end{cor}
\begin{proof}
Applying propositions \ref{tens}, \ref{com} and then \ref{tens} again we obtain
\[
\begin{split}
F_{1*} m \otimes F_{2*} n &= F_{1*}( m \otimes F_1^*F_{2*} n) = F_{1*}( m \otimes \Pi_{1*}\Pi_2^*n) \\ 
&= F_{1*}\Pi_{1*} (\Pi_1^* m \otimes \Pi_2^* n).
\end{split}
\]
\end{proof}
Note that $F_{1*} m \otimes F_{2*} \chi_{P^2} = F_{1*}\Pi_{1*} (\Pi_1^* m \otimes \Pi_2^* \chi_{P^2}) = F_{1*}\Pi_{1*} (\Pi_1^* m \otimes \chi_{P}) = F_{1*}\Pi_{1*} \Pi_1^* m$. In particular $F_{1*} \chi_{P^1} \otimes F_{2*} \chi_{P^2} = F_{1*}\Pi_{1*}\chi_{P}$.

\begin{exmp} Let $P^1$ and $P^2$ be subquivers of a quiver $Q$. Further let $F_i : \Bbbk P^i \rightarrow \Bbbk Q$ be given by inclusion for $i \in \{1,2\}$. It follows that $F_{i*}\chi_{P^i} = \chi_{P^i}^{Q}$.

The fibre product $P = P^1 \times P^2$ is isomorphic to the intersection $P^1 \cap P^2$ and the wrappings $\Pi_i : \Bbbk P \rightarrow \Bbbk P^i$ for $i \in \{1,2\}$ correspond to the inclusions via this isomorphism. By Corollary \ref{lift}, $F_{1*}\chi_{P^1}\otimes F_{2*}\chi_{P^2} \iso F_{1*}\Pi_{1*}\chi_{P^1\cap P^2}$. In other words $\chi_{P^1}^{Q} \otimes \chi_{P^2}^{Q} \iso \chi_{P^1\cap P^2}^{Q}$. This result is proved directly in \cite{herschend07b}.
\end{exmp}

A wrapping $F:\Bbbk P \rightarrow \Bbbk Q$ is called strict if for all pairs of distinct arrows $\arr x \alpha y $ and $\arr {x'} {\alpha'} {y'} $ in $P$, $F\alpha = F\alpha'$ implies $x \not = x'$ and $y \not = y'$. In this case we call the pair $\mathbf{F} = (F, P)$ a shape over $Q$. For example each subquiver of $Q$ gives rise to a shape over $Q$ via inclusion. Note that compositions of strict wrappings are strict wrappings. In particular the restriction of a strict wrapping is strict. 

Let $\mathbf{F_1} = (F_1, P^1)$ and $\mathbf{F_2} = (F_2, P^2)$ be shapes over $Q$. We say that they are isomorphic (via $\sigma : \Bbbk P^1 \rightarrow \Bbbk P^2$) if $\sigma$ is defined by a quiver isomorphism $P^1 \iso P^2$ and satisfies $F_1 = F_2 \circ \sigma$. We denote the isomorphism class of a shape $\mathbf{F}$ by $\overline{\mathbf{F}}$. 

Let $P = P^1 \times P^2$. Since $F_2$ is strict, so is $\Pi_1$ and by symmetry $\Pi_2$. Hence we obtain a strict wrapping $F_1\Pi_2 = F_2\Pi_1 : \Bbbk P \rightarrow \Bbbk Q$. We denote the shape $(F_1\Pi_2, P)$ by $\mathbf{F_1}\times\mathbf{F_2}$.

Let $R$ be the disjoint union of $P^1$ and $P^2$. The sum of $\mathbf{F_1} $ and $\mathbf{F_2}$ is defined to be the shape $\mathbf{F_1} + \mathbf{F_2}= (F, R)$ such that $F$ restricted to $P^1$ and $P^2$ is equal to $F_1$ and $F_2$ respectively.

If $P^1= P^2$ and $F_1 = F_2$, then the diagonal of $P$ is the subquiver $\Delta$ defined by $\Delta_0 = \{(x,x) \;|\; x \in P^1_0\}$ and $\Delta_1 = \{(\alpha,\alpha) \;|\; \alpha \in P^1_1\}$. Its compliment is the full subquiver of $(P \times P)_0 \setminus \Delta_0$. This terminology is justified by the following result.

\begin{pro}\label{diag}
Let $(F,P)$ be a shape over $Q$. Then the diagonal $\Delta$ is a component in $P \times P$, \ie $P \times P$ is the disjoint union of $\Delta$ its compliment.
\end{pro}
\begin{proof}
If $(\alpha, \alpha') \in (P \times P)((x,x),(y,y'))$ then $F(\alpha) = F(\alpha')$. But $t \alpha = x = t \alpha'$ and so, since $F$ is strict $\alpha = \alpha'$ and $y = y'$. The case $(\alpha, \alpha') \in (P \times P)((x,x'),(y,y))$ is similar.
\end{proof}

Let $\mathbf{F} = (F,P)$ be a shape over $Q$. By Proposition \ref{diag}, the quiver $P \times P$ is the disjoint union of the diagonal $\Delta$ and its complement $R$. Hence $\mathbf{F} \times \mathbf{F} = (F\Pi_1|_\Delta, \Delta) + (F\Pi_1|_R, R)$. Moreover, the canonical isomorphism $\Delta \iso P$ defines an isomorphism $(\Pi_1|_\Delta, \Delta) \iso (Id_P,P)$. Hence $(F\Pi_1|_\Delta, \Delta) \iso \mathbf{F}$.

\section{String categories}
In this section we apply the methods developed in the previous section to find the decomposition of tensor products of strings and bands, and thus solve the Clebsch-Gordan problem for string algebras. We state everything for small linear categories as opposed to algebras.
 
Throughout the rest of the paper let $Q$ be a quiver and $\Lambda = \Bbbk Q / \mathcal{I}$ for some semi-monomial ideal $\mathcal{I}$ in $\Bbbk Q$. Furthermore, assume that all shapes over $Q$ are compatible with $\mathcal{I}$. We say that $\Lambda$ is a string category if its morphism spaces are finite dimensional and the following conditions are satisfied:

\begin{enumerate}
\item The ideal $\mathcal{I}$ is generated by a set of paths.

\item Each vertex $x \in Q_0$ is the tail respectively head, of at most two arrows in $Q$, \ie $|g^{-1}x| \leq 2$ for $g \in \{t,h\}$.

\item For each arrow $\alpha \in Q_1$ there is at most one $\beta \in Q_1$ such that $\alpha \beta \not \in \mathcal{I}$ and at most one $\gamma \in Q_1$ such that $\gamma \alpha \not \in \mathcal{I}$.
\end{enumerate}

We proceed to introduce strings and bands. A quiver is called linear if it is of type Dynkin type $\mathbb{A}$, \ie it has underlying graph
\[
\xymatrix{\b \edge{r} & \cdots \edge{r}& \b }
\]
and cyclic if it is of extended Dynkin type $\tilde{\mathbb{A}}$, \ie it has underlying graph
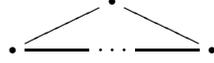
\[
\xymatrix@R=9pt{& \b \edge{dl} &\\ \b \edge{r} & \cdots \edge{r}& \b \edge{ul}}
\]
A shape $\mathbf{F} = (F,L)$ over $Q$ is called linear if $L$ is linear. The string associated to a linear shape $\mathbf{F}$ is the $\Lambda $-module $S_\mathbf{F} = F_*\chi_L$.

A shape $\mathbf{G} = (G,Z)$ over $Q$ is called cyclic if $Z$ is cyclic and $\mathbf{G}$ has trivial automorphism group. Let $\lambda \in \Bbbk^{\iota}$, $s >0$ and $\gamma \in Z_1$. Define the $\Bbbk Z$-module $B = B(\lambda, s, \gamma)$ by $B(x) = \Bbbk^s$ for each $x \in Z_0$, $B(\alpha) = 1_{\Bbbk^s}$ for $\alpha \not = \gamma$ and $B(\gamma) = J_{\lambda}(s)$. The band associated to $(\mathbf{G},\lambda, s, \gamma)$ is the $\Lambda $-module $B_\mathbf{G}(\lambda, s,\gamma) = G_*B$. 

Let $\gamma' \in Z_1$. We say that $\gamma$ and $\gamma'$ are oriented equally if when cycling through the vertices of $Z$ we encounter $t\gamma$ and $h\gamma$ in the same order as we encounter $t\gamma'$ and $h\gamma'$. In that case $B_\mathbf{G}(\lambda, s,\gamma') \iso B_\mathbf{G}(\lambda, s,\gamma)$. Otherwise $B_\mathbf{G}(\lambda, s,\gamma') \iso B_\mathbf{G}(\lambda^{-1}, s,\gamma)$. 

The following classification of indecomposable modules over string categories follows from \cite{wald83}.

\begin{thm}\label{wald} 
Assume that $\Lambda$ is a string category. Then the strings and bands classify all indecomposable $\Lambda$-modules, \ie 
\begin{enumerate}
\item Every string and band is indecomposable.
\item Each indecomposable $\Lambda$-module is isomorphic to either a string or band.
\item No strings are isomorphic to bands. 
\item Two strings $S_F$ and $S_{F'}$ are isomorphic if and only if they have of isomorphic shapes.
\item Two bands $B_\mathbf{G}(\lambda, s,\gamma)$ and $B_{\mathbf{G}'}(\mu, t, \gamma')$ are isomorphic if and only if $s = t$ and their shapes are isomorphic via some  $\sigma$ such that $\lambda = \mu^r$, where $r = 1$ if $\sigma(\gamma)$ and $\gamma'$ are equally oriented and $r = -1$ otherwise.
\end{enumerate}
\end{thm}

Observe that cyclic and linear quivers themselves define string categories. Thus the push-down functors corresponding to strings and bands in fact map the isoclasses of indecomposables to the isoclasses of indecomposables in $\cat{mod}\Lambda$.

To solve the Clebsch-Gordan problem for string categories it suffices to find the decomposition of tensor products of strings and bands by Theorem \ref{wald}. To do this we have to determine the fibre products of linear and cyclic shapes. We start with a small example to illustrate our strategy.

\begin{exmp}\label{lorentz} Assume that $Q$ is the double loop quiver
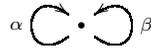
\[
\xymatrix{\b \ar@(dl,ul)^{\alpha}\ar@(dr,ur)_{\beta}}
\]
and that $\mathcal{I}$ is generated by the paths $\alpha^n$, $\beta^n$, $\alpha\beta$ and $\beta\alpha$. Consider the following two linear shapes (where the labelling indicates where the arrows are mapped rather than the arrows themselves):
\[
\xymatrix{ \mathbf{F}: & \b \ar[r]^{\alpha}& \b \ar[r]^{\alpha}& \b & \b \ar[l]_{\beta} \ar[r]^{\alpha}& \b \\
\mathbf{F'}: & \b & \b \ar[l]_{\beta} \ar[r]^{\alpha} & \b \ar[r]^{\alpha}& \b & \b \ar[l]_{\beta}& \b \ar[l]_{\beta}
}
\]
Their fibre product can be visualised in a rectangle together with $\mathbf{F}$ written vertically on the left and $\mathbf{F'}$ written horizontally below:

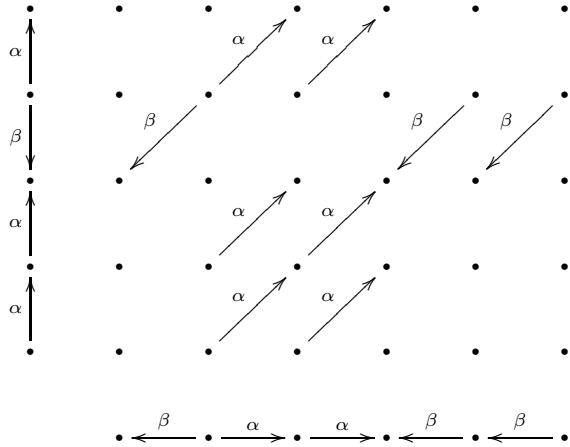
\[
\xymatrix{
 \b 
 & \b & \b & \b & \b & \b & \b \\
 \b\ar[u]^{\alpha}\ar[d]_{\beta} 
 & \b & \b\ar[ur]^{\alpha}\ar[dl]_{\beta}  & \b\ar[ur]^{\alpha} & \b & \b\ar[dl]_{\beta}  & \b\ar[dl]_{\beta}  \\
 \b 
 & \b & \b & \b & \b & \b & \b \\
 \b\ar[u]^{\alpha} 
 & \b & \b\ar[ur]^{\alpha} & \b\ar[ur]^{\alpha} & \b & \b & \b \\
 \b\ar[u]^{\alpha} 
 & \b & \b\ar[ur]^{\alpha} & \b\ar[ur]^{\alpha} & \b & \b & \b \\
& \b & \b\ar[l]_{\beta}\ar[r]^{\alpha} & \b\ar[r]^{\alpha} & \b & \b \ar[l]_{\beta}& \b\ar[l]_{\beta} \\
}
\]
Hence its connected components are linear shapes. More precisely, there are three unique components with the following respective shapes:
\[
\xymatrix{ \mathbf{F}_{\alpha\alpha\beta}: & \b \ar[r]^{\alpha}& \b \ar[r]^{\alpha}& \b & \b \ar[l]_{\beta} \\
\mathbf{F}_{\beta\alpha}: & \b & \b \ar[l]_{\beta} \ar[r]^{\alpha} & \b \\
\mathbf{F}_{\beta}: & \b & \b \ar[l]_{\beta}.
}
\]
Moreover, there are $3$ copies of $\mathbf{F}_\alpha : \arr \b \alpha \b$ and $15$ copies of the shape $\mathbf{F}_0$, which consists of a single vertex and no arrows. By Corollary \ref{lift},
\[
S_\mathbf{F} \otimes S_\mathbf{F'} \iso S_{\mathbf{F}_{\alpha\alpha\beta}} \oplus S_{\mathbf{F}_{\beta\alpha}} \oplus S_{\mathbf{F}_{\beta}} \oplus 3S_{\mathbf{F}_\alpha} \oplus 15S_{\mathbf{F}_0}.
\]
\end{exmp}
Our aim is to generalise this example to include any pair of linear or cyclic shapes for arbitrary $Q$ and $\mathcal{I}$. To do this we need some more terminology.

Let $\mathbf{F} = (F,P)$ and $\mathbf{F'} = (F',P')$ be shapes over $Q$. We say that $\mathbf{F'}$ factors through $\mathbf{F}$ if there is a strict wrapping $G: \Bbbk P' \rightarrow \Bbbk P$ such that the diagram
\[
\xymatrix{ & \Bbbk P' \ar[d]^{F'}\ar[dl]_{G} \\
\Bbbk P \ar[r]_{F} & \Bbbk Q}
\]
commutes. If there is such a $G$, that in addition is a monomorphism (\ie given by a monomorphism of quivers $P' \rightarrow P$) we say that $\mathbf{F'}$ is a subshape of $\mathbf{F}$. We make the following straightforward but never the less useful observation.

\begin{lem}\label{sub}
Let $\mathbf{F}= (F,P)$ be a linear shape over $Q$. If $\mathbf{F'} = (F',P')$ is a connected shape that factors through $\mathbf{F}$, then $\mathbf{F'}$ is a subshape of $\mathbf{F}$ and in particular $\mathbf{F'}$ is linear.
\end{lem}
\begin{proof}
Consider a strict wrapping $G: \Bbbk P' \rightarrow \Bbbk P$. We claim that $G$ is a monomorphism. If $G$ maps $P'_0$ injectively into $P_0$ we are done since $G$ is a wrapping. Otherwise there is a linear subquiver of $P'$ connecting two vertices mapped to the same vertex in $P$. In this subquiver there must be two arrows violating the condition that $G$ is strict.
\end{proof}

Let $\mathbf{F} = (F,P)$ and $\mathbf{F'} = (F',P')$ be shapes over $Q$. Furthermore, let $\{P^j\}_{j \in J}$ be the set of connected components of $P \times P'$ and set $\mathbf{H}_j = (F\Pi_1|_{P^j}, P^j)$. Then the fibre product $\mathbf{F}\times \mathbf{F'} = \sum_{j \in J}\mathbf{H}_j$.  By abuse of language we call the shapes $\mathbf{H}_j$ connected components of $\mathbf{F}\times \mathbf{F'}$. Let $\mathcal{L}(\mathbf{F}, \mathbf{F'})$ be the set of those $\mathbf{H}_j$ that are linear. If $\mathbf{F} =\mathbf{F'}$, then the summand corresponding to the diagonal is denoted $\mathbf{\Delta}$.

\begin{pro}\label{fiber}
Let $\mathbf{F}$ and $\mathbf{F'}$ be shapes over $Q$. Set $\mathbf{G} = \mathbf{F}\times \mathbf{F'}$. Then the following statements hold
\begin{enumerate}
\item\label{fiber1} Each connected component of $\mathbf{G}$ factors through both $\mathbf{F}$ and $\mathbf{F'}$.

\item\label{fiber2} If one of $\mathbf{F}$ and $\mathbf{F'}$ is linear or if both are cyclic but not isomorphic then
\[
\mathbf{G} = \sum_{\mathbf{H} \in \mathcal{L}(\mathbf{F}, \mathbf{F'})} \mathbf{H}.
\]

\item \label{fiber3} If $\mathbf{F}$ is cyclic, then
\[
\mathbf{F} \times \mathbf{F} = \mathbf{\Delta} + \sum_{\mathbf{H} \in \mathcal{L}(\mathbf{F}, \mathbf{F})}\mathbf{H}.
\]
\end{enumerate}
\end{pro}
\begin{proof}
Each connected component $\mathbf{H}=(H, R)$ of $\mathbf{G}$ factors through $\mathbf{F}$ and $\mathbf{F'}$ via $\Pi_1|_R$ and $\Pi_2|_R$ respectively. This proves statement \ref{fiber1}.

If one of $\mathbf{F}$ and $\mathbf{F'}$ is linear, then each connected component $\mathbf{H}$ of $\mathbf{G}$ factors through a linear wrapping by statement \ref{fiber1} and thus $\mathbf{H}$ is also linear by Lemma \ref{sub}.

Now assume that $\mathbf{F} = (F,Z)$ and $\mathbf{F'} = (F',Z')$ are cyclic. Let $\mathbf{H}=(H, R)$ be a connected component of $\mathbf{G}$. Consider the strict wrapping $\Pi_1|_R:\Bbbk R \rightarrow \Bbbk Z$. Let $z \in R_0$ and $x = \Pi_1z$. Since $\Pi_1$ is strict $|g^{-1}z| \leq |g^{-1}x|$ for $g \in \{t,h\}$. Hence $|t^{-1}z| + |h^{-1}z| \leq |t^{-1}x| + |h^{-1}x| = 2$. It follows that $R$ is linear or cyclic.

Assume that $R$ is cyclic. Then $R$ has the underlying graph
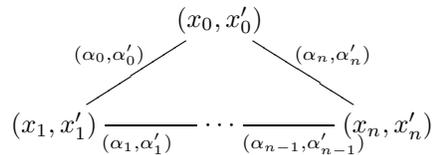
\[
\xymatrix{& (x_0,x_0') \edge{dl}_{(\alpha_0,\alpha_0')} &\\ (x_1,x_1') \edge{r}_{(\alpha_1,\alpha_1')} & \cdots \edge{r}_{(\alpha_{n-1},\alpha_{n-1}')}& (x_n,x_n') \edge{ul}_{(\alpha_n,\alpha_n')}}
\] 
Let $k$ and $l$ be the number of vertices of $Z$ and $Z'$ respectively. We claim that $k \leq n +1$ and $\alpha_i = \alpha_j $ if and only if $i \equiv j \mod k$. If $Z$ is a loop this is certainly true. 

Assume that $Z$ is not a loop. Also, set $\alpha_{n+1} = \alpha_0$ and $x_{n+1} = x_0$. If $\alpha_i = \alpha_{i+1}$ for some $0 \leq i \leq n$, then $\Pi_1(\alpha_i,\alpha_i') = \Pi_1(\alpha_{i+1},\alpha_{i+1}')$, and since $\Pi_1|_R:\Bbbk R \rightarrow \Bbbk Z$ is strict $(\alpha_i,\alpha_i') = (\alpha_{i+1},\alpha_{i+1}')$. Thus $R$ is a loop which is a contradiction since $Z$ is not a loop. 

For all $0 \leq i \leq n$ the two distinct arrows $\alpha_i$ and $\alpha_{i+1}$ are connected via the vertex $x_{i+1}$. Since $Z$ is cyclic it follows that the sequence of arrows $\alpha_0, \ldots, \alpha_n$ contains each arrow in $Z$ and repeats exactly at every $k$ steps. Hence $k \leq n +1$ and $\alpha_i = \alpha_j $ if and only if $i \equiv j \mod k$. By symmetry $l \leq n +1$ and $\alpha_i' = \alpha_j' $ if and only if $i \equiv j \mod l$.

Assume $l < n+1$. If $i \equiv j \mod l$, then $F(\alpha_i) = F'(\alpha_i') = F'(\alpha_{j}') = F(\alpha_{j})$ and $F(x_i) = F'(x_i') = F'(x_{j}') = F(x_{j})$. Hence we obtain an automorphism $\sigma: Z \rightarrow Z$ sending $\alpha_0$ to $\alpha_l$ such that $F \iso F \circ \sigma$. It is trivial since $\mathbf{F}$ has trivial automorphism group. Thus $\alpha_0 = \alpha_l$ and $(\alpha_0,\alpha_0') = (\alpha_{l},\alpha_{l}')$ which is a contradiction. Consequently, $l = n+1$ and by symmetry $k = n+1$.

Now we obtain an isomorphism $\sigma : Z \rightarrow Z'$ defined by $x_i \mapsto x_i'$ and $\alpha_i \mapsto \alpha_i'$. If no such isomorphism exists, then all connected components of $Z \times Z'$ are linear and statement \ref{fiber2} follows.

On the other hand, if $Z = Z'$ and $F = F'$, then $\sigma$ must be the identity since $\mathbf{F}$ has trivial automorphism group. Hence $R$ is the diagonal of $Z \times Z$ and all other connected components are linear. This yields statement \ref{fiber3}.
\end{proof}

\begin{thm}\label{CG}
For linear shapes $\mathbf{F}=(F,L)$, $\mathbf{F'}=(F',L')$, non-isomorphic cyclic shapes $\mathbf{G}=(G, Z)$, $\mathbf{G'}=(G', Z')$, scalars $\lambda, \mu \in \Bbbk^{\iota}$ and positive integers $s, t$, the following formulae hold:
\[\begin{split}
S_\mathbf{F} \otimes S_\mathbf{F'} &\iso \bigoplus_{\mathbf{H}\in \mathcal{L}(\mathbf{F}, \mathbf{F'})} S_\mathbf{H} 
\\
S_\mathbf{F} \otimes B_\mathbf{G}(\lambda, s, \gamma) &\iso \bigoplus_{\mathbf{H}\in \mathcal{L}(\mathbf{F}, \mathbf{G})} s S_\mathbf{H}
\\
B_\mathbf{G}(\lambda, s, \gamma) \otimes B_\mathbf{G'}(\mu, t, \gamma') &\iso \bigoplus_{\mathbf{H}\in \mathcal{L}(\mathbf{G}, \mathbf{G'})} st S_\mathbf{H}
\\
B_\mathbf{G}(\lambda, s, \gamma) \otimes B_\mathbf{G}(\mu, t, \gamma) &\iso \bigoplus_{k \in I_{st}} B_\mathbf{G}(\lambda\mu, l_k, \gamma) \oplus \bigoplus_{\mathbf{H}\in \mathcal{L}(\mathbf{G}, \mathbf{G})} st S_\mathbf{H}
\end{split}\]
\end{thm}
\begin{proof}
Let $\mathbf{F_1} = (F_1,P^1)$ and $\mathbf{F_2} = (F_2,P^2)$ be shapes over $Q$. Further let $m_1 \in \cat{mod} \Bbbk P^1$ and $m_2 \in \cat{mod} \Bbbk P^2$. By Corollary \ref{lift},
\[
F_{1*} m_1 \otimes F_{2*} m_2 = F_{1*}\Pi_{1*} (\Pi_1^* m_1 \otimes \Pi_2^* m_2).
\]
Let $\{\mathbf{H}_j = (H_j,P^j)\}_{j \in J}$ be the connected components of $\mathbf{F_1} \times \mathbf{F_2}$ and  consider the module $m = \Pi_1^* m_1 \otimes \Pi_2^* m_2$. It decomposes into a direct sum
\[
m = \bigoplus_{j \in J} m_j,
\]
where $m_j$ has support $P^j$. Moreover, for each $(x_1,x_2) \in P^j_0$ and $(\alpha_1,\alpha_2) \in P^j_1$ we have $m_j(x_1,x_2) = m_1(x_1) \otimes m_2(x_2)$ and $m_j(\alpha_1,\alpha_2) = m_1(\alpha_1) \otimes m_2(\alpha_2)$. 

Assume that $P^j$ is linear. If all $m_1(\alpha_1)$ and $m_2(\alpha_2)$ are isomorphisms, then each vector space $m_1(x_1) \otimes m_2(x_2)$ has the same dimension $d$. From the classification of indecomposable $\Bbbk P^j$-modules it follows that $m_j \iso d \chi_{P^j}$ and 
\[
F_{1*}\Pi_{1*}m_j \iso d S_{\mathbf{H}_j}.
\]

Assume that $\mathbf{F_1} = \mathbf{F_2} = \mathbf{G}$, $m_1 = B(\lambda, s,\gamma)$, $m_2 = B(\mu, t,\gamma)$ and $P^j$ is the diagonal of $Z \times Z$. Then 
\[
\Pi_{1*} m_j = B(\lambda, s,\gamma) \otimes B(\mu, t,\gamma) \iso \bigoplus_{k \in I_{st}} B(\lambda\mu, l_k,\gamma),
\]
by the solution to the Clebsch-Gordan problem for cyclic quivers \cite{herschend05}. Hence 
\[
F_{1*}\Pi_{1*}m_j \iso \bigoplus_{k \in I_{st}}  B_\mathbf{G}(\lambda\mu, l_k,\gamma).
\]
The theorem now follows from Proposition \ref{fiber}.
\end{proof}
If two bands have isomorphic shapes they may be assumed have equal shapes by Theorem \ref{wald}. Hence Theorem \ref{CG} takes into account all relevant cases.

By Theorem \ref{CG}, the Clebsch-Gordan problem for string categories is reduced to determining $\mathcal{L}(\mathbf{F},\mathbf{F'})$ for each pair of linear or cyclic shapes $(\mathbf{F},\mathbf{F'})$. By Proposition \ref{fiber}, each linear shape $\mathbf{H} \in \mathcal{L}(\mathbf{F},\mathbf{F'})$ must factor through both $\mathbf{F}$ and $\mathbf{F'}$ (which for linear $\mathbf{F}$ (or $\mathbf{F'}$) means that $\mathbf{H}$ is a subshape by Lemma \ref{sub}). To find a completely explicit solution to the Clebsch-Gordan problem for string categories it remains to determine with what multiplicity each isomorphism class of linear shapes factoring through $\mathbf{F}$ and $\mathbf{F'}$ appears in $\mathcal{L}(\mathbf{F},\mathbf{F'})$. In each specific case this combinatorial task is quite easy to carry out, \eg as in Example \ref{lorentz}. To obtain better understanding of the general case we investigate the representation ring.

\section{Representation ring}\label{repring}
In this section we assume that $\Lambda$ is a string category. Note that rings are not assumed to have identity elements in this section. The representation ring $R(\Lambda)$ is the Grothendieck ring associated with the semi-ring of isoclasses of $\Lambda$-modules (with addition and multiplication given by direct sum and tensor product respectively). For a more detailed definition see \cite{herschend07b}. As an abelian group, $R(\Lambda)$ is freely generated by the isoclasses of indecomposables, \ie strings and bands. It has an identity element if and only if $\chi_Q$ annihilates $\mathcal{I}$, which in our setting implies that $\mathcal{I}$ is the zero ideal and the connected components of $Q$ are linear or cyclic.

Our description of $R(\Lambda)$ will depend on the representation ring of the Laurent algebra $R(\Gamma)$, which is described in \cite{herschend08a}. If $\Bbbk$ is of characteristic zero, then $R(\Gamma) \iso \mathbb{Z}\Bbbk^{\iota}[T]$, the polynomial ring in one variable with coefficients in the group ring of the group of invertible elements $\Bbbk^{\iota}$. 

Let $I$ be the $\mathbb{Z}$-span in $R(\Lambda)$ of all isoclasses of strings. By Theorem \ref{wald}, $I$ has a $\mathbb{Z}$-basis consisting of the elements $[S_\mathbf{F}]$ where $\overline{\mathbf{F}}$ ranges through $\mathcal{L}$, the set all isoclasses of linear shapes compatible with $\mathcal{I}$. By Theorem \ref{CG}, $I$ is an ideal in $R(\Lambda)$. 

To handle relationships between shapes over $Q$ that determine the formulae in Theorem \ref{CG}, define the partial order $\leq$ on $\mathcal{L}$ by $\overline{\mathbf{F'}} \leq \overline{\mathbf{F}}$ if $\mathbf{F'}$ is a subshape of $\mathbf{F}$. The information contained in this partial order can be refined as follows.

For every pair of shapes $\mathbf{F} = (F,P)$ and $\mathbf{F'} = (F',P')$ over $Q$ denote by $[\mathbf{F}: \nolinebreak\mathbf{F'}]$, the set of strict wrappings $G: \Bbbk P' \rightarrow \Bbbk P$ such that $F' = FG$. Furthermore, let $|\mathbf{F}:\mathbf{F'}|$ be the number of elements in $[\mathbf{F}:\mathbf{F'}]$. Thus $|\mathbf{F}:\mathbf{F'}|$ counts the number of ways that $\mathbf{F'}$ factors through $\mathbf{F}$.

For all $\overline{\mathbf{F}},\overline{\mathbf{F'}} \in \mathcal{L}$ the set $[\mathbf{F}:\mathbf{F'}]$ is non-empty if and only if $\overline{\mathbf{F}} \geq \overline{\mathbf{F'}}$. Moreover, $[\mathbf{F}:\mathbf{F}]$ is a singleton. 

\begin{lem}\label{pull}
Let $\mathbf{F_1} = (F_1,P^1)$, $\mathbf{F_2} = (F_2,P^2)$ and $\mathbf{F} = (F,P)$ be shapes over $Q$. Also, assume that $P$ is connected. Then there are bijections
\[
[\mathbf{F_1}:\mathbf{F}] \times [\mathbf{F_2}:\mathbf{F}] \rightarrow [\mathbf{F_1}\times \mathbf{F_2}:\mathbf{F}].
\]
\[
[\mathbf{F_1}:\mathbf{F}] \amalg [\mathbf{F_2}:\mathbf{F}] \rightarrow [\mathbf{F_1} + \mathbf{F_2}:\mathbf{F}].
\]
\end{lem}
\begin{proof}
We claim that for each pair of elements $G_1 \in [\mathbf{F_1}:\mathbf{F}]$ and $G_2 \in [\mathbf{F_2}:\mathbf{F}]$ there is a unique strict wrapping $G: \Bbbk P \rightarrow \Bbbk (P^1 \times P^2)$ such that the diagram
\[
\xymatrix@C=9pt{ & \Bbbk (P^1 \times P^2) \ar[dl]_{\Pi_1}\ar[dr]^{\Pi_2} & \\
\Bbbk P^1 \ar[dr]_{F_1} & \Bbbk P\ar[d]_{F} \ar[r]^{G_2}\ar[l]_{G_1}\ar[u]^{G}& \Bbbk P^2\ar[dl]^{F_2}\\ 
& \Bbbk Q &
}
\]
commutes. In particular there is a bijection $[\mathbf{F_1}:\mathbf{F}] \times [\mathbf{F_2}:\mathbf{F}] \rightarrow [\mathbf{F_1}\times \mathbf{F_2}:\mathbf{F}$].

Let $x \in P_0$ and $\alpha \in P_1$. The commutativity of the diagram is equivalent to $G(x) = (G_1(x), G_2(x))$ and $G(\alpha) = (G_1(\alpha), G_2(\alpha))$, which uniquely determines $G$. This $G$ is well-defined since $F_1G_1(x) = F(x) = F_2G_2(x)$ and $F_1G_1(\alpha) = F(\alpha) = F_2G_2(\alpha)$. Since $G_1$ is strict, so is $G$.

Since $P$ is connected, each element in $[\mathbf{F_1} + \mathbf{F_2}:\mathbf{F}]$ must have its image in $\Bbbk P^1$ or $\Bbbk P^2$. Hence there is a bijection $[\mathbf{F_1}:\mathbf{F}] \amalg [\mathbf{F_2}:\mathbf{F}] \rightarrow [\mathbf{F_1} + \mathbf{F_2}:\mathbf{F}]$. 
\end{proof}

Now let $\{\mathbf{H}_j\}_{j \in J}$ be the connected components of $\mathbf{F_1}\times \mathbf{F_2}$. By Lemma \ref{pull}, we have
\begin{equation}\label{mult}
|\mathbf{F_1}:\mathbf{F}| |\mathbf{F_2}:\mathbf{F}| = |\mathbf{F_1}\times \mathbf{F_2}:\mathbf{F}| = \left|\left(\sum_{j \in J} \mathbf{H}_j\right):\mathbf{F}\right| = \sum_{j \in J}| \mathbf{H}_j:\mathbf{F}|.
\end{equation}
The following theorem is the main result of this section. 

\begin{thm}\label{ring}
The ideal $I$ has a unique $\mathbb{Z}$-basis of pair-wise orthogonal idempotents $\{e_\mathbf{F} = e_{\overline{\mathbf{F}}}\}_{\overline{\mathbf{F}} \in \mathcal{L}}$, such that the following statements hold:
\begin{enumerate}
\item\label{ring1} For each linear shape $\mathbf{F}$
\[
[S_\mathbf{F}] = \sum_{\overline{\mathbf{F'}} \in \mathcal{L}} | \mathbf{F}:\mathbf{F'} | e_\mathbf{F'} = 
\sum_{\overline{\mathbf{F'}} \leq \overline{\mathbf{F}}} | \mathbf{F}:\mathbf{F'} | e_\mathbf{F'}.
\]
Moreover, $[S_\mathbf{F}]e_\mathbf{F'} = |\mathbf{F} : \mathbf{F'}| e_\mathbf{F'}$.

\item\label{ring2} For each cyclic shape $\mathbf{G} = (G, Z)$, $\gamma \in Z_1$, $\lambda \in \Bbbk^{\iota}$ and positive integer $s$
\[
[B_\mathbf{G}(\lambda, s,\gamma)]e_\mathbf{F'} = s |\mathbf{G} : \mathbf{F'}| e_\mathbf{F'}.
\]
\item\label{ring3} For each pair of non-isomorphic cyclic shapes $\mathbf{G_1} = (G_1, Z^1)$, $\mathbf{G_2} = (G_2, Z^2)$, $\gamma_1 \in Z^1_1$, $\gamma_2 \in Z^2_1$, $\lambda, \mu \in \Bbbk^{\iota}$ and positive integers $s,t$
\[
[B_\mathbf{G_1}(\lambda, s,\gamma_1)][B_\mathbf{G_2}(\mu, t,\gamma_2)] = \sum_{\overline{\mathbf{F'}} \in \mathcal{L}} s t |\mathbf{G_1}:\mathbf{F'}||\mathbf{G_2} : \mathbf{F'}| e_\mathbf{F'}.
\]
Moreover,
\[\begin{split}
[B_\mathbf{G_1}(\lambda, s,\gamma_1)][B_\mathbf{G_1}(\mu,t,\gamma_1)] &=
\sum_{k \in I_{st}}[B_\mathbf{G_1}(\lambda\mu, l_k ,\gamma_1)] + \\
&\sum_{\overline{\mathbf{F'}} \in \mathcal{L}} s t |\mathbf{G_1} : \mathbf{F'}| (|\mathbf{G_1} : \mathbf{F'}| -1) e_\mathbf{F'}.
\end{split}\]
\end{enumerate}
\end{thm}
\begin{proof}
Let $S$ be the ring having the $\mathbb{Z}$-basis $\{d_{\mathbf{F}}=d_{\overline{\mathbf{F}}}\}_{\overline{\mathbf{F}} \in \mathcal{L}}$, and multiplication defined by $d_{\overline{\mathbf{F}}}d_{\overline{\mathbf{F'}}} = \delta_{\overline{\mathbf{F}},\overline{\mathbf{F'}}}d_{\overline{\mathbf{F}}}$. Define the $\mathbb{Z}$-linear map $\theta : I \rightarrow S$ by $\theta[S_\mathbf{F}] = \sum_{\overline{\mathbf{F'}} \in \mathcal{L}} | \mathbf{F}:\mathbf{F'} | d_\mathbf{F'}$. It is well-defined since $| \mathbf{F}:\mathbf{F'} | = 0$ for all $\mathbf{F'} \not \leq \mathbf{F}$. 

Observe that
\[
\theta ([S_\mathbf{F_1}][S_\mathbf{F_2}]) = \sum_{\mathbf{H}\in \mathcal{L}(\mathbf{F_1}, \mathbf{F_2})} \theta [S_\mathbf{H}] = \sum_{\mathbf{H}\in \mathcal{L}(\mathbf{F_1}, \mathbf{F_2})}\sum_{\overline{\mathbf{F'}} \in \mathcal{L}}
|\mathbf{H}:\mathbf{F'}| d_\mathbf{F'}.
\]
By equation (\ref{mult}) we have that 
\[
\sum_{\mathbf{H}\in \mathcal{L}(\mathbf{F_1}, \mathbf{F_2})}
|\mathbf{H}:\mathbf{F'}| =  |\mathbf{F_1}:\mathbf{F'}| |\mathbf{F_2}:\mathbf{F'}|,
\]
and thus
\[\begin{split}
\theta[S_\mathbf{F_1}]\theta[S_\mathbf{F_2}] &= 
\left(\sum_{\overline{\mathbf{F'}} \in \mathcal{L}} | \mathbf{F_1}:\mathbf{F'} | d_\mathbf{F'}\right)
\left(\sum_{\overline{\mathbf{F'}} \in \mathcal{L}} | \mathbf{F_2}:\mathbf{F'} | d_\mathbf{F'}\right) \\
&= \sum_{\overline{\mathbf{F'}} \in \mathcal{L}} |\mathbf{F_1}:\mathbf{F'}| |\mathbf{F_2}:\mathbf{F'}| d_\mathbf{F'} = 
\theta ([S_\mathbf{F_1}][S_\mathbf{F_2}]).
\end{split}\]
Hence $\theta$ is a ring morphism. Since $| \mathbf{F}:\mathbf{F} | = 1$ we have
\[
\theta[S_\mathbf{F}] = d_\mathbf{F} + \sum_{\overline{\mathbf{F'}} < \overline{\mathbf{F}}} | \mathbf{F}:\mathbf{F'} | d_\mathbf{F'}.
\]
Define the $\mathbb{Z}$-linear map $\rho: S \rightarrow I$ recursively by
\[
\rho d_\mathbf{F} = [S_\mathbf{F}] - \sum_{\overline{\mathbf{F'}} < \overline{\mathbf{F}}} | \mathbf{F}:\mathbf{F'} | \rho d_\mathbf{F'}.
\]
We show by induction that $\theta\rho d_\mathbf{F} = d_\mathbf{F}$:
\[\begin{split}
\theta \rho d_\mathbf{F} &= \theta [S_\mathbf{F}] - \sum_{\overline{\mathbf{F'}} < \overline{\mathbf{F}}} | \mathbf{F}:\mathbf{F'} | \theta\rho d_\mathbf{F'} \\
&= d_\mathbf{F} + \sum_{\overline{\mathbf{F'}} < \overline{\mathbf{F}}} | \mathbf{F}:\mathbf{F'} | d_\mathbf{F'} - \sum_{\overline{\mathbf{F'}} < \overline{\mathbf{F}}} | \mathbf{F}:\mathbf{F'} | d_\mathbf{F'} = d_\mathbf{F}.
\end{split}\]
We also have
\[
\rho\theta[S_\mathbf{F}] = \rho d_\mathbf{F} + \sum_{\overline{\mathbf{F'}} < \overline{\mathbf{F}}} | \mathbf{F}:\mathbf{F'} | \rho d_\mathbf{F'} = [S_\mathbf{F}].
\]
Hence $\theta$ and $\rho$ are inverse to each other. Choosing $e_\mathbf{F} = \rho d_\mathbf{F}$ we see that statement \ref{ring1} is satisfied. Moreover, statement \ref{ring1} determines the elements $e_\mathbf{F}$ uniquely since they must satisfy the recurrence relation $e_\mathbf{F} = [S_\mathbf{F}] - \sum_{\overline{\mathbf{F'}} < \overline{\mathbf{F}}} | \mathbf{F}: \nolinebreak\mathbf{F'} | e_\mathbf{F'}$. Furthermore, note that
\[
[S_\mathbf{F}]e_\mathbf{F''} =  \sum_{\overline{\mathbf{F'}} \in \mathcal{L}} | \mathbf{F}:\mathbf{F'} | e_\mathbf{F'}e_\mathbf{F''} = | \mathbf{F}:\mathbf{F''}| e_\mathbf{F''}.
\]

Now consider $[B_\mathbf{G}(\lambda, s,\gamma)]e_\mathbf{F}$. If $\mathbf{F}$ has no proper linear subshapes, then $e_\mathbf{F} = [S_\mathbf{F}]$. Also, $\mathcal{L}(\mathbf{G},\mathbf{F})$ consists of
$|\mathbf{G}:\mathbf{F}|$ linear shapes all isomorphic to $\mathbf{F}$ and thus
\[
[B_\mathbf{G}(\lambda, s,\gamma)]e_\mathbf{F} = s|\mathbf{G}:\mathbf{F}|e_\mathbf{F}.
\]
We show by induction that $[B_\mathbf{G}(\lambda, s,\gamma)]e_\mathbf{F}= s |\mathbf{G}:\mathbf{F}| e_\mathbf{F}$ for all $\overline{\mathbf{F}} \in \mathcal{L}$:
\[\begin{split}
[B_\mathbf{G}(\lambda, s,\gamma)]e_\mathbf{F} 
&= [B_\mathbf{G}(\lambda, s,\gamma)][S_\mathbf{F}] - \sum_{\overline{\mathbf{F'}} < \overline{\mathbf{F}}} | \mathbf{F}:\mathbf{F'} | [B_\mathbf{G}(\lambda, s,\gamma)]e_\mathbf{F'} \\
&= \sum_{\mathbf{H}\in \mathcal{L}(\mathbf{G}, \mathbf{F})} s [S_\mathbf{H}] -
\sum_{\overline{\mathbf{F'}} < \overline{\mathbf{F}}} 
s | \mathbf{F}:\mathbf{F'} | |\mathbf{G}:\mathbf{F'}| e_\mathbf{F'} \\
&= s\left(
\sum_{\mathbf{H}\in \mathcal{L}(\mathbf{G}, \mathbf{F})}
\sum_{\overline{\mathbf{F'}} \in \mathcal{L}} | \mathbf{H}:\mathbf{F'} | 
e_\mathbf{F'} -\sum_{\overline{\mathbf{F'}} < \overline{\mathbf{F}}} | \mathbf{F}:\mathbf{F'} | |\mathbf{G}:\mathbf{F'}| e_\mathbf{F'} 
\right) \\
&= s\left(
\sum_{\overline{\mathbf{F'}} \in \mathcal{L}}
\left|\left(\sum_{\mathbf{H}\in \mathcal{L}(\mathbf{G}, \mathbf{F})}\mathbf{H}\right):\mathbf{F'} \right| 
e_\mathbf{F'} -\sum_{\overline{\mathbf{F'}} < \overline{\mathbf{F}}} | \mathbf{F}:\mathbf{F'} | |\mathbf{G}:\mathbf{F'}| e_\mathbf{F'} 
\right) \\
&= s\left(
\sum_{\overline{\mathbf{F'}} \in \mathcal{L}} | \mathbf{G}\times \mathbf{F}:\mathbf{F'} | 
e_\mathbf{F'} -\sum_{\overline{\mathbf{F'}} < \overline{\mathbf{F}}} | \mathbf{F}:\mathbf{F'} | |\mathbf{G}:\mathbf{F'}| e_\mathbf{F'} 
\right) \\
&= s\left(
\sum_{\overline{\mathbf{F'}} \leq \overline{\mathbf{F}}} |\mathbf{G}:\mathbf{F'}||\mathbf{F}:\mathbf{F'}|
e_\mathbf{F'} -\sum_{\overline{\mathbf{F'}} < \overline{\mathbf{F}}} | \mathbf{F}:\mathbf{F'} | |\mathbf{G}:\mathbf{F'}| e_\mathbf{F'} 
\right) \\
&= s |\mathbf{G}:\mathbf{F}||\mathbf{F}:\mathbf{F}| e_\mathbf{F} = s |\mathbf{G}:\mathbf{F}| e_\mathbf{F}.
\end{split}\]
Thus statement \ref{ring2} is proved.

To prove statement \ref{ring3}, write
\[
[B_\mathbf{G_1}(\lambda, s,\gamma_1)][B_\mathbf{G_2}(\mu, t,\gamma_2)] = \sum_{\overline{\mathbf{F'}} \in \mathcal{L}} \lambda_\mathbf{F'}e_\mathbf{F'},
\]
which is possible by Theorem \ref{CG}. Then for each $\overline{\mathbf{F}} \in \mathcal{L}$
\[
[B_\mathbf{G_1}(\lambda, s,\gamma_1)]e_\mathbf{F}[B_\mathbf{G_2}(\psi,\gamma_2)]e_\mathbf{F} = st |\mathbf{G_1}:\mathbf{F}| |\mathbf{G_2}:\mathbf{F}| e_\mathbf{F}.
\]
Hence
\[
\lambda_\mathbf{F} = st |\mathbf{G_1}:\mathbf{F}| |\mathbf{G_2}:\mathbf{F}|.
\]

Similarly write
\[
[B_\mathbf{G_1}(\lambda, s,\gamma_1)][B_\mathbf{G_1}(\mu, t,\gamma_1)] = \sum_{k \in I_{st}}[B_\mathbf{G_1}(\lambda\mu, l_k ,\gamma_1)] + \sum_{\overline{\mathbf{F'}} \in \mathcal{L}} \mu_\mathbf{F'}e_\mathbf{F'}.
\]
For each $\overline{\mathbf{F}} \in \mathcal{L}$
\[
[B_\mathbf{G_1}(\lambda, s,\gamma_1)]e_\mathbf{F}[B_\mathbf{G_1}(\mu, t,\gamma_1)]e_\mathbf{F} = st |\mathbf{G_1}:\mathbf{F}|^2 e_\mathbf{F}.
\]
On the other hand
\[\begin{split}
[B_\mathbf{G_1}(\lambda, s,\gamma_1)][B_\mathbf{G_1}(\mu,t,\gamma_1)]e_\mathbf{F} &= 
\sum_{k \in I_{st}}[B_\mathbf{G_1}(\lambda\mu, l_k ,\gamma_1)]e_\mathbf{F} + \mu_\mathbf{F}e_\mathbf{F} \\
&= (st |\mathbf{G_1}:\mathbf{F}| + \mu_\mathbf{F}) e_\mathbf{F}.
\end{split}\]
We obtain
\[
\mu_\mathbf{F} = st |\mathbf{G_1}:\mathbf{F}|(|\mathbf{G_1}:\mathbf{F}|-1).
\]
which yields statement \ref{ring3}. 
\end{proof}

Let $\mathcal{B}$ be the set of isoclasses of cyclic shapes.
\begin{cor}\label{poly}
In the notation above
\[
R(\Lambda)/I \iso \bigoplus_{\overline{\mathbf{G}} \in \mathcal{B}} R(\Gamma).
\]
\end{cor}
\begin{proof}
For each $\overline{\mathbf{G}} = \overline{(H, Z)}\in \mathcal{B}$ fix some arrow $\gamma \in Z_1$. Define a $\mathbb{Z}$-linear map $R(\Lambda)/I \rightarrow \bigoplus_{\overline{\mathbf{G}} \in \mathcal{B}} R(\Gamma)$ by sending $[B_\mathbf{G}(\lambda, s,\gamma)]$ to $[b_{\lambda,s}]$ in the copy of $R(\Gamma)$ corresponding to $\mathbf{G}$. That this defines an isomorphism follows from Theorem \ref{ring}.
\end{proof}

By Corollary \ref{poly}, the ring $R(\Lambda)/I$ is isomorphic to the direct sum of a certain number of copies of $R(\Gamma)$. By Theorem \ref{ring}, the ideal $I$ has a canonical $\mathbb{Z}$-basis of orthogonal idempotents. To get a complete description of $R(\Lambda)$ it remains to describe the action of bands on $I$ and to determine the direct summand in the tensor product of two bands that lies in $I$. By Theorem \ref{ring}, this can be achieved by computing the numbers $|\mathbf{G}:\mathbf{F}|$, for all $\overline{\mathbf{G}} \in \mathcal{B}$ and $\overline{\mathbf{F}} \in \mathcal{L}$. 

For instance, if $\mathbf{G} = (G,Z)$ is such that $G$ is a monomorphism (which is always the case for $Q$ of type $\tilde{\mathbb{A}}$, where in fact $G$ can be taken to be the identity) then $|\mathbf{G}:\mathbf{F}| = 1$ for all $\mathbf{F}$ whose image lie in the image of $\mathbf{G}$ and $|\mathbf{G}:\mathbf{F}| = 0$ otherwise. In particular, for all $\lambda, \mu \in \Bbbk^{\iota}$, $s,t>0$ and $\gamma \in Z_1$ we obtain $[B_\mathbf{G}(\lambda, s,\gamma)][B_\mathbf{G}(\mu, t,\gamma)] = \sum_{k\in I_{st}}[B_\mathbf{G}(\lambda \mu, l_k,\gamma)]$. Moreover, if $\mathbf{F}$ factors through $\mathbf{G}$, then $[B_\mathbf{G}(\lambda, s,\gamma)]e_\mathbf{F} = s e_\mathbf{F}$. Otherwise $[B_\mathbf{G}(\lambda, s,\gamma)]e_\mathbf{F} = 0$

\section{Tensor ideals}
In the previous section we treated the Clebsch-Gordan problem for string categories by investigating the representation ring. Another approach is to consider so-called tensor ideals. These have been studied in other settings, usually for triangulated categories by various authors, \eg \cite{benson97} and \cite{balmer05}. In this section we will consider tensor ideals in $\cat{mod} \Lambda$, where they also make sense.

Let $X_\Lambda$ be the set of all isoclasses of indecomposable $\Lambda$-modules. For any subcategory $\mathcal{C}$ of $\cat{mod} \Lambda$, let $[\mathcal{C}]$ be the smallest full subcategory of $\cat{mod} \Lambda$ closed under isomorphisms and containing $\mathcal{C}$.

A full subcategory $\mathcal{U} \subset \cat{mod} \Lambda$ is called a tensor ideal if it is closed under isomorphisms, direct sums, direct summands and satisfies $m \otimes u \in \mathcal{U}$ for all $m \in \cat{mod} \Lambda$ and $u \in \mathcal{U}$. In particular $\mathcal{U}$ is uniquely determined by the set $\overline{\mathcal{U}} \subset X_\Lambda$ of isoclasses of indecomposables in $\mathcal{U}$.

Let $\mathcal{U}$ and $\mathcal{V}$ be tensor ideals. Their sum $\mathcal{U} + \mathcal{V}$ is the smallest tensor ideal containing both $\mathcal{U}$ and $\mathcal{V}$. It consists of all elements of the form $u \oplus v$, where $u \in \mathcal{U}$ and $v \in \mathcal{V}$. Moreover, $\overline{\mathcal{U} + \mathcal{V}} = \overline{\mathcal{U}} \cup \overline{\mathcal{V}}$. This notion generalises to sums of arbitrary cardinality.

For each $x \in X_\Lambda$ we let $\mathcal{N}_{\Lambda}(x) = \mathcal{N}(x)$ be the smallest tensor ideal containing $x$ and $N_{\Lambda}(x) = N(x) = \overline{\mathcal{N}(x)}$. Let $\mathcal{U}$ be a tensor ideal. Then $\mathcal{U} = \sum_{x \in \overline{\mathcal{U}}} \mathcal{N}(x)$ or equivalently $\overline{\mathcal{U}} = \bigcup_{x \in \overline{\mathcal{U}}} N(x)$. Hence the subsets of $X_\Lambda$, which correspond to tensor ideals can be recovered if we know $N(x)$ for all $x \in X_\Lambda$. We proceed to determine the sets $N(x)$.

First we treat the case $\Lambda = \Gamma$. For all $\lambda \in \Bbbk^{\iota}$ and $s >0$ let $N_{\lambda,s} = N_{\Gamma}(b_{\lambda, s})$. If $\Bbbk$ has characteristic zero, then $N_{\lambda,s}$ consists of all isoclasses of indecomposables in $\cat{mod} \Gamma$ by Theorem \ref{jordan}. Indeed, in this case the tensor identity $b_{1, 1}$ is a direct summand of $b_{\lambda, s} \otimes b_{\lambda^{-1}, s}$. In prime characteristic $N_{\lambda,s} = N_{1,s}$ by \cite{herschend08a}.

Let $P$ be a quiver and $F : \Bbbk P \rightarrow \Bbbk Q$ a wrapping compatible with $\mathcal{I}$. Furthermore, let $X_P$ be the set of isoclasses of indecomposable $\Bbbk P$-modules. Assume that $F_{*}$ reflects indecomposability. In other words $F_*$ induces a map $f: X_{P} \rightarrow X_\Lambda$.

\begin{lem}\label{image}
Let $\mathcal{U} \subset \cat{mod} \Bbbk P$ and $\mathcal{V} \subset \cat{mod} \Lambda$ be tensor ideals. Then the following statements hold.
\begin{enumerate}
\item The essential image $[F_* \mathcal{U}] \subset \cat{mod} \Lambda$ is a tensor ideal.
\item The preimage $F_*^{-1} \mathcal{V} \subset \cat{mod} \Bbbk P$ is a tensor ideal.
\end{enumerate}
\end{lem}
\begin{proof}
Since $F_{*}$ reflects indecomposability and $\mathcal{U}$, $\mathcal{V}$ are tensor ideals, both $[F_* \mathcal{U}]$ and $F_*^{-1} \mathcal{V}$ are closed under direct sums and direct summands. 

Let $u \in \mathcal{U}$ and $m \in \cat{mod} \Lambda$. By Proposition \ref{tens}, $F_{*}u \otimes m \iso F_{*}(u \otimes F^*m) \in [F_* \mathcal{U}]$ and thus $[F_* \mathcal{U}]$ is a tensor ideal.

Now let $v \in F_*^{-1} \mathcal{V}$ and $m \in \cat{mod} \Bbbk P$. Then $F_* v \otimes F_* m \iso F_{*}\Pi_{1*} (\Pi_1^* v \otimes \Pi_2^* m)$ by Proposition \ref{lift}. Consider the $\Bbbk (P \times P)$-module $\Pi_1^* v \otimes \Pi_2^* m$. By Lemma \ref{diag}, it decomposes into the direct sum of $n$ with support in $\Delta$, and $n'$ with support in the complement of $\Delta$. For each $x \in P_0$ and $\alpha \in P_1$ we have $n(x,x) = v(x) \otimes m(x)$ and $n(\alpha,\alpha) = v(\alpha) \otimes m(\alpha)$. Hence $\Pi_{1*}n \iso v \otimes m$ and thus $F_{*}(v \otimes m)$ is a direct summand of $F_* v \otimes F_* m$. Since $\mathcal{V}$ is a tensor ideal, $F_{*}(v \otimes m) \in \mathcal{V}$ and thus $v \otimes m \in F_*^{-1} \mathcal{V}$. 
\end{proof}

\begin{pro}\label{smallest}
Let $x \in X_P$. Then the map $f$ induces a surjection $N_{\Bbbk P}(x) \rightarrow N_\Lambda(f x)$.
\end{pro}
\begin{proof}
Let $m$ be a representative of $x$. By Lemma \ref{image}, the essential image $[F_* \mathcal{N}_{\Bbbk P}(x)]$ is a tensor ideal containing $F_* m$, and thus $\mathcal{N}_\Lambda(f x) \subset [F_* \mathcal{N}_{\Bbbk P}(x)]$. Moreover, the preimage $F_*^{-1} \mathcal{N}_\Lambda(f x)$ is a tensor ideal containing $m$. Hence $\mathcal{N}_{\Bbbk P}(x) \subset F_*^{-1} \mathcal{N}_\Lambda(f x)$ and $\mathcal{N}_\Lambda(f x) \supset [F_* \mathcal{N}_{\Bbbk P}(x)]$. It follows that the map $f : X_{\Bbbk P} \rightarrow X_\Lambda$ induces a surjection $N_{\Bbbk P}(x) \rightarrow N_\Lambda(f x)$.
\end{proof}

\begin{thm}\label{tensideal}
Assume that  $\Lambda$ is a string category. Let $\mathbf{F} = (F,L)$ be a linear shape and $\mathbf{G} = (G,Z)$ a cyclic shape. Let $\lambda \in \Bbbk^{\iota}$, $s >0$ and $\gamma \in Z_1$. Then the following statements hold.
\begin{enumerate}
\item The set $N([S_\mathbf{F}])$ consists of the isoclasses of strings whose shapes are subshapes of $\mathbf{F}$.
\item The set $N([B_\mathbf{G}(\lambda, s, \gamma)])$ consists of the isoclasses of strings whose shapes factor through $\mathbf{G}$ and $[B_\mathbf{G}(\mu, t, \gamma)]$ where $t$ is such that $[b_{1,t}] \in N_{1,s}$.
\end{enumerate}
\end{thm}
\begin{proof}
Since $\chi_L \otimes n \iso n$ for each $\Bbbk L$-module $n$ we have $\mathcal{N}_{\Bbbk L}([\chi_L]) = \cat{mod} \Bbbk L$. 

Let $B(\lambda, s,\gamma) \in \cat{mod} \Bbbk Z$. For each linear shape $\mathbf{F'} = (F',L')$ over $Z$ we have $B(\lambda, s,\gamma)  \otimes S_\mathbf{F'} \iso s S_\mathbf{F'}$ and thus $[S_\mathbf{F'}]  \in N_{\Bbbk Z}([B(\lambda, s,\gamma)])$. Moreover, $B(\lambda, s,\gamma) \otimes B(\mu, t,\gamma) \iso \sum_{k \in I_{st}} B(\lambda\mu, l_k, \gamma)$. Hence the bands in $N_{\Bbbk Z}([B(\lambda, s,\gamma) ])$ are precisely those $[B(\mu, t, \gamma)]$ where $t$ is such that $[b_{1,t}] \in N_{1,s}$. The theorem now follows from Proposition \ref{smallest}.
\end{proof}

\bibliographystyle{plain}
\bibliography{biban}

\vspace*{1cm}
\noindent
\begin{tabular}[t]{l@{}}{\it Martin Herschend}\vspace*{0.2cm}\\
{\it Graduate School of Mathematics}\\
{\it Nagoya University, Chikusa-ku}\\
{\it Nagoya, 464-8602}\\
{\it Japan}\vspace*{0.2cm} \\
{\tt martinh@math.nagoya-u.ac.jp}
\end{tabular}

\end{document}